\documentclass[12pt,a4paper]{amsart}
\usepackage{amscd,amssymb,amsopn,amsmath,amsthm,graphics,amsfonts,enumerate,verbatim,calc}
\usepackage{amsmath}
\usepackage{amssymb}  
\usepackage{amsthm}
\usepackage{microtype}
\usepackage{mathrsfs}
\usepackage{tikz}
\usetikzlibrary{shapes,arrows,positioning}

\usepackage[colorlinks=true,linkcolor=blue,citecolor=blue]{hyperref}
\headsep=1cm \numberwithin{equation}{section}
\newtheorem{theorem}{Theorem}[section]

\newtheorem{lemma}[theorem]{Lemma}
\newtheorem{corollary}[theorem]{Corollary}
\newtheorem{remark}[theorem]{Remark}
\newtheorem{example}[theorem]{Example}
\newtheorem{definition}[theorem]{Definition}
\newcommand{\field}[1]{\mathbb{#1}}
\newcommand{\R}{\field{R}} 
\newcommand{\N}{\field{N}} 
\newcommand{\prob}{\field{P}}

\newcommand{\expect}{\field{E}}

\numberwithin{equation}{section}

\begin{document}

\title{A Transience Criterion For Uniformly Bounded Markov Chains With Asymptotically Zero Mean Drift}
\author{Dan Andrei Tudor}

\address{Master's Student in Applied Mathematics, Technische Universiteit Delft, Netherlands}
 \email { datudor@tudelft.nl}



\begin{abstract}
In this paper, we give an overview of mean drift conditions for the state-space classification of discrete-time Markov Chains and we present a new transience criterion for uniformly bounded Markov Chains with asymptotically zero drift. The criterion does not need a condition on the second-moment drifts and can be applied to certain chains for which other criteria fail. \vspace{0.2 cm} \\
{\bf Mathematics Subject Classification (2010):} 60J10  	\\
{\bf Key words:} markov chains, mean drift conditions, Foster-Lyapunov functions
\end{abstract}

\maketitle

\section{Introduction and notation}
 Throughout the present paper, we shall look at a Markov Chain $(X_n)_{n=0}^{\infty} $with state space $S=\N_0:=\N\cup\{0\}$ and transition probabilities $(p_{ij})_{i,j \in \N_0}$, which we shall colloquially refer to as \textit{the chain}. We furthermore assume that the chain is \textit{irreducible}.
\par According to \cite{norris}, a state $i$ is recurrent if $\prob(T_i < \infty | X_0=i)=1$ and transient otherwise, where $T_i:=\inf\{n \geq 1 : X_n=i\}$. It is called \textit{positive} recurrent if it is recurrent and  $\expect(T_i | X_0=i)<\infty$ and it is called \textit{null} recurrent if it is recurrent and $\expect(T_i| X_0=i)=\infty$. In the case of an \textit{irreducible} Markov Chain, the states are either all \textit{(positive/null) recurrent} or all \textit{transient}.
\par The scope of the present paper is to describe the chain's structure with the aid of  \textit{mean~drifts} and to additionally develop a new mean~drift criterion for transience, under certain regularity conditions.
Section~2 is devoted to showcasing the four main Foster-Lyapunov type criteria for the classification of the chain. These theorems involve the use of a Lyapunov-type function $f$, which needs to be chosen suitably. In the next section, we present the concept of\textit{ mean drifts} and characterize the chain's behaviour with their aid, when they are not asymptotically zero. We develop a simple criterion for null recurrence, which allows us to construct non-trivial examples of null recurrent chains. The final section is devoted to analysing the chain's behaviour when the mean drifts are asymptotically zero. Here, we showcase a new criterion for \textit{transience} when the chain is uniformly bounded and provide a family of mean drifts for which this criterion is applicable. 

\section{Foster-Lyapunov criteria}

We start this section by stating without proof four important general theorems that regard the stability of the chain. All three theorems make use of the incremental change of the Lyapunov function $f$ between two states.
\begin{theorem}[Foster \cite{foster}]\label{fosterthm} The chain is positive recurrent if there exists some non-negative function $f:\N_0 \rightarrow \R_{\geq0}$ and $N \geq 1$ such that
\begin{align}
    &\sum_{j \in \N_0} p_{ij}(f(j)-f(i)) \leq \epsilon &&& \textup{for } i \geq N,\label{notinafoster} \\  
   &\sum_{j \in \N_0} p_{ij} f(j) < \infty &&& \textup{for } i< N, \label{inafoster}
\end{align}
for some $\epsilon>0$.
\end{theorem}
The following corollary will be of use in constructing examples in the following sections.

\begin{corollary}\label{awayfromzero}
    If there exists $j \in \N_0$ such that $\inf_{i \in \N_0/\{j\}} p_{ij}>0$, then the chain is positive~recurrent.
\end{corollary}

\begin{proof}
    By reindexing the states, we can assume without loss of generality that $j=0$. Denote $\alpha:=\inf_{i\in \N} p_{ij}>0$. In Foster's theorem, take $N=1$, $f(0)=0$ and $f(i)=1$ for $i \geq 1$, and $\epsilon=\alpha$. Inequality \eqref{inafoster} is immediately satisfied for this function.
    We have for $i > 0$,
    \begin{equation*}
        p_{i0}f(0)+\sum_{k>0} p_{ik} f(k) =1-p_{ij} \leq 1-\alpha,
    \end{equation*}
     showing \eqref{notinafoster}. Therefore, the conditions of Foster's theorem are satisfied and the chain is positive recurrent.
\end{proof}

The following two theorems complement each other nicely. Note the additional constraints on $f$ in the second one.

\begin{theorem}[Fayolle-Malyshev-Menshikov \cite{fayolle}] \label{fayolle-transience}
    The chain is transient if there exists some non-negative function $f$ such that for some $N \geq 1$
    \begin{align}
       &\sum_{j \in \N_0} p_{ij}(f(j)-f(i)) \geq \epsilon &&& \textup{for } i \geq N
    \end{align}
    for some $\epsilon>0$, and for some $d>0$, $|f(i)-f(j)| > d$ implies that $p_{ij}=0$.
\end{theorem}
\begin{remark}
    The last condition requires a sort of uniform bound on the growth of $f$. This condition is gonna play an important role later on in the special case when $f(i)=i$.
\end{remark}

\begin{theorem}[Mertens et al. \cite{mertens}]
\label{transient-mertens}
    The chain is transient if there exists some bounded non-constant function $f: \N_0 \rightarrow \R$ such that for some $N \geq 1$
    \begin{align} \label{transientineq}
        &\sum_{j \in S} p_{ij} (f(j) - f(i)) \leq 0 &&& \textup{for } i \geq N,
    \end{align} and for some $k \geq N$ it holds that $f(k) < f(i)$ for all $i \leq N$.
\end{theorem}

\medskip

Finally, we present below the condition for recurrence of the chain.

\begin{theorem}[Mertens et al. \cite{mertens}]\label{recurrent-mertens}
    The chain is recurrent if there exists some function $f:\N_0 \rightarrow \R$ such that for some $N\geq 1$
\begin{align}
    &\sum_{j \in \N_0} p_{ij}(f(j)-f(i)) \leq 0 &&& \textup{for } i \geq N,
\end{align}
and $\lim_{i \rightarrow \infty} f(i) = \infty$.

\end{theorem}

\section{Mean drifts and behaviour away from zero}

Let us now turn our attention to the concept of \textit{mean drift}, which is defined for every state $i \in \N_0$. Intuitively, the \textit{mean drift} at state $i$ tells you whether the chain is expected to jump upward (positive drift), downward (negative drift), or stay in the same state (zero drift) when located at $i$.

\begin{definition}
    The \textit{mean drift} at state $i$ is defined as $\gamma_i:=\expect(X_{n+1}-X_n|X_n=i)=\sum_{j\in \N_0} p_{ij}(j -i)$.
\end{definition}

Note that $-i\leq \gamma_i \leq +\infty$ for all $i \in \N_0$. Using the intuitive idea that a negative drift tends to pull the chain downward, Pakes  \cite{pakes} gives the following sufficient condition for positive recurrence.

\begin{lemma}[Pakes' Lemma \cite{pakes}] \label{pakeslemma}
    The Markov Chain is positive recurrent if $\gamma_i<\infty$ for all $i \in \N_0$ and $\limsup_{i \rightarrow \infty} \gamma_i < 0$.
\end{lemma}
\begin{proof}
    The last condition implies that there exists $\varepsilon>0$ and $N \geq 1$ such that $\gamma_i \leq -\varepsilon$ for all $i \geq N$. In other words, we have $ \sum_{j \in \N_0} p_{ij}(j-i) \leq -\epsilon$, for all $i\geq N$. Since all drifts are finite, we also have $\sum_{j \in \N_0}p_{ij} \, j < \infty$ for all $i \in \N_0$. Therefore, the conditions of Foster's theorem are satisfied with $f(i)=i$ and the chain is positive recurrent. 
\end{proof}

Similarly, one can conclude recurrence  of the chain when the drifts become eventually nonpositive, as illustrated by the following lemma.

\begin{lemma}
    The chain is recurrent if there exists $N\geq1$ such that $\gamma_i \leq 0$ for all $i \geq N$.
\end{lemma}

\begin{proof}
    The conclusion follows from Theorem~\ref{recurrent-mertens} applied with $f(i)=i$.
\end{proof}

We now turn our attention to those Markov Chains whose mean drift at each state is (eventually) non-negative. Contrary to intuition, one can construct positve recurrent Markov Chains for which the mean drift even diverges to $+ \infty$. The following classical example \cite{sennott} illustrates this.

\begin{example}\label{ergodicdrift}
    Let the chain be governed by the transition probabilities $p_{0i}=\frac{1}{2^i}$, and $p_{i0}=p_{i,3i}=\frac{1}{2}$ for all $i \geq 1$. The point of allowing transitions from $0$ to any other state is to make the chain irreducible. Note that $\inf_{i \in \N} p_{i0}=1/2>0$, therefore, by Corollary~\ref{awayfromzero}, the chain is positive recurrent. However,  for $i >0$,
 \begin{align*}
     &\gamma_i=\frac{i}{2} \rightarrow \infty \quad\textup{as } i \rightarrow \infty.
 \end{align*}
 \end{example}

 It becomes clear then that the chain must behave in a more specific way in order to classify its structure when the mean drifts are non-negative. The issue in the example above is that we allow the chain to revisit the zero state from any other state. If, instead, we restrict the downward movement of the chain only to "nearby" states, we can better analyse the chain's behaviour using mean
 drifts. This restriction is illustrated in the following definition.

\begin{definition}
    A Markov Chain is called \textit{downward uniformly bounded} if there exists $k \in \N$ such that $i-j>k$ implies $p_{ij}=0$.
\end{definition}
 Kaplan's criterion \cite{kaplan, sennott} then gives non-ergodicity of the chain when the mean drifts are (eventually) non-negative and the chain is downward uniformly bounded.

\begin{theorem}[Kaplan's criterion for non-ergodicity]\label{kaplancriterion}
    The chain is \underline{not} ergodic if it is downward uniformly bounded, $\gamma_i<\infty$ for all $i\geq 0$, and, for some $N \geq 1$, $\gamma_i\geq 0$ for all $i \geq N$.
\end{theorem}

\begin{remark}
    In the original paper \cite{kaplan}, Kaplan introduces what we shall call as \textit{Kaplan's function} $K_i:[0,1)\rightarrow \R$ with $K_i(z)=(z^i-\sum_j p_{ij}z^j)/(1-z)$. Instead of requiring the chain to be downward uniformly bounded, Kaplan's original condition imposes that there~exists $B \geq 0$, $N \in \N$ and $c \in [0,1)$ such that $K_i(z)\geq -B$ for all $i \geq N$ and $z \in [c,1)$. Kaplan's original condition exists solely to justify the application of Fatou's lemma at one stage of the proof of this theorem. However, since this condition is not easy to check, we look specifically at downward uniformly bounded chains. We show in the following lemma that downward uniform boundedness of the chain implies that Kaplan's original condition holds.
\end{remark}

\begin{lemma}
    For a downward uniformly bounded chain, Kaplan's original condition holds.
\end{lemma}

\begin{proof}
    Suppose the chain is downward uniformly bounded, that is, there exists a $k \in \N$ such that $p_{ij}=0$ for all $j<i-k$ and all $i \in \N_0$. Then, we have, for all $c \in (0,1)$ and $z \in [c,1)$,
   \begin{align*}
       z^i-\sum_jp_{ij}z^j&=\sum_{j \geq i-k}p_{ij}(z^i-z^{j})\\
       &\geq\sum_{j=i-k}^{i-1}p_{ij}(z^i-z^j)=-(1-z)\sum_{j=i-k}^{i-1}p_{ij}\frac{z^j-z^i}{1-z}\\
       &=-(1-z)\sum_{j=i-k}^{i-1}p_{ij}z^{j}\frac{1-z^{i-j}}{1-z}=-(1-z)\sum_{j=i-k}^{i-1}p_{ij}z^j\sum_{m=0}^{i-j+1}z^m,
       \end{align*}
       where the last equality follows from $1-z^n=(1-z)(1+\dots+z^{n-1})$. Observe that $z^j\sum_{m=0}^{i-j+1}z^m \leq (i-j)$, as $z<1$. Thus,
       \begin{align*}
           z^i-\sum_jp_{ij}z^j&\geq-(1-z)\sum_{j=i-k}^{i-1}p_{ij}(i-j)=-(1-z)\sum_{j=1}^{k}p_{i,i-j}j\\&\geq -(1-z)\sum_{j=1}^k j=-(1-z)\frac{k(k+1)}{2}.
       \end{align*}
       Thus, Kaplan's original condition holds with $B=k(k+1)/2$.
\end{proof}

Additionally from Kaplan's criterion for non-ergodicity, we can separately present two sufficient conditions, one for null recurrence and one for transience. In particular,  for downward uniformly bounded chains, if the mean drifts are eventually zero, then the chain is null recurrent.

\begin{lemma}
\label{nullrecurrence}
    If the chain is downward uniformly bounded, $\gamma_i<\infty$ for all $i\geq 0$, and, for some $N \in \N$, $\gamma_i= 0$ for~all~$i \geq N$, then it is null recurrent.
\end{lemma}
\begin{proof}
    Observe that the conditions for Kaplan's criterion hold, therefore the chain is non-ergodic. To see that it is recurrent, we apply Theorem~\ref{recurrent-mertens} with $f(i)=i$. Since the condition in the theorem is equivalent in terms of mean drifts with $\gamma_i \leq 0$ for all $i \geq N $, we conclude the chain is recurrent, and thus null recurrent.
\end{proof}

Let us now see two examples of classes of chains
that are null recurrent. This easy-to-check condition allows us to immediately conclude the chains are null-recurrent, without going through lengthy computations.

\begin{example}\label{sitstill}
    Let $k \in \N$. Consider a random walk on $\N_0$ in which the chain can "stay still" at a certain state. The transition probabilities are $p_{i,i-k}=\dotso=p_{i,i+k}=\frac{1}{2k+1}$ for~$i \geq k$ and $p_{i,i+1}=1$ for $i<k$.
    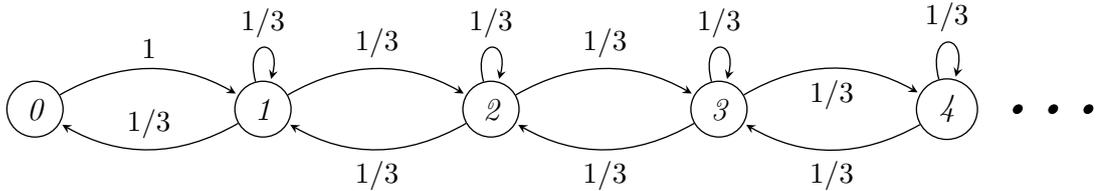
\begin{figure}[h!]
 \centering
 \begin{tikzpicture}[->,>=stealth ,shorten >=1pt, line width=0.5pt,
 node distance=3cm]
 \node [circle, draw] (zero) {0};
 \node [circle, draw] (one) [right of=zero] {1};
 \node [circle, draw] (two) [right of=one] {2};
  \node [circle, draw] (three) [right of=two] {3};
  \node [circle, draw] (four) [right of=three] {4};
 \path (zero) edge [bend left] node[above] {\small{$1$}} (one);
 \path (one) edge[bend left] node[above] {\small{$1/3$}} (zero);
  \path (one) edge [loop above] node[above] {\small{$1/3$}} (one);
 \path (one) edge[bend left] node[above]{\small{$1/3$}} (two);
 \path (two) edge[bend left] node[below]{\small{$1/3$}} (one);
   \path (two) edge [loop above] node[above] {\small{$1/3$}} (two);
 \path (two) edge[bend left] node[above]{\small{$1/3$}} (three);
 \path (three) edge[bend left] node[below]{\small{$1/3$}} (two);
  \path (three) edge [loop above] node[above] {\small{$1/3$}} (three);
 \path (four) edge[bend left] node[below]{\small{$1/3$}} (three);
 \path (three) edge[bend left] node[below]{\small{$1/3$}} (four);
   \path (four) edge [loop above] node[above] {\small{$1/3$}} (four);

 \node[right=2mm of four,font=\bfseries\Huge]{\dots};        
 \end{tikzpicture}
 \caption{State transition diagram of Example~\ref{sitstill} for $k=1$}
 \label{fig:randomwalk}
 \end{figure}
\par The chain is downward uniformly bounded as $p_{ij}=0$ for $j<i-k$. Note that $\gamma_0=~\dotso~=~\gamma_{k-1}=1$ and $\gamma_i=0$ for $i \geq k$. Therefore, by Lemma~\ref{nullrecurrence}, the chain is null~recurrent.
    
\end{example}

\begin{example}\label{lopsided}
    Consider two natural numbers $m,n \in \N$ with $\gcd(m,n)=1$. Consider the chain with transition probabilities $p_{i,i+n}=\frac{m}{m+n}$ for all $i$, $p_{i,i-m}=\frac{n}{m+n}$ for $i \geq m$ and $p_{ii}=\frac{n}{m+n}$ for $i<m$.
    \begin{figure}[h!]
 \centering
 \begin{tikzpicture}[->,>=stealth ,shorten >=1pt, line width=0.5pt,
 node distance=2cm]
 \node [circle, draw] (zero) {0};
 \node [circle, draw] (one) [right of=zero] {1};
 \node [circle, draw] (two) [right of=one] {2};
  \node [circle, draw] (three) [right of=two] {3};
  \node [circle, draw] (four) [right of=three] {4};
  \node[circle, draw] (five) [right of=four] {5};
 \path (zero) edge [bend left] node[above] {\small{$1/3$}} (two);
 \path (zero) edge [loop above] node[above] {\small{$2/3$}} (zero);
 \path (one) edge[bend left] node[below] {\small{$2/3$}} (zero);
  \path (one) edge [bend left] node[above] {\small{$1/3$}} (three);
 \path (two) edge[bend left] node[below]{\small{$2/3$}} (one);
 \path (two) edge[bend left] node[above]{\small{$1/3$}} (four);
 \path (three) edge[bend left] node[below]{\small{$2/3$}} (two);
  \path (three) edge [bend left] node[above] {\small{$1/3$}} (five);
 \path (four) edge[bend left] node[below]{\small{$2/3$}} (three);
    \path (five) edge[bend left] node[below]{\small{$2/3$}} (four);
 \node[right=2mm of five,font=\bfseries\Huge]{\dots};        
 \end{tikzpicture}
 \caption{State transition diagram of Example~\ref{lopsided} for $m=1$ and $n=2$}
 \label{fig:randomwalk}
 \end{figure}
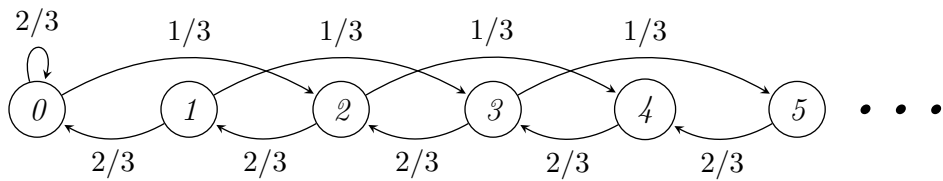
    
    \par Observe that the chain is downward uniformly bounded. For $i \geq m$, we have \begin{equation*}
        \gamma_i=\frac{m}{m+n}n+\frac{n}{m+n}(-m)=0.
    \end{equation*} 
   Therefore, the chain is null recurrent by Lemma~\ref{nullrecurrence}.
   \par In particular, note that, for $m=n=1$, the chain behaves as a simple symmetric random walk, which is known to be null recurrent. However, proving null recurrence for the simple symmetric random walk is a rather involved task by elementary means and usually relies on large Sterling approximations, see for instance \cite[Example~1.6.1]{norris}.
\end{example}

We can also similarly construct an equivalent statement to Pakes' Lemma for transience for those chains which are uniformly bounded also from above, which we define below.

\begin{definition}
    A Markov Chain is called \textit{uniformly bounded} if there exists $d \in \N$ such that $|i-j|>d$ implies $p_{ij}=0$.
\end{definition}

\begin{lemma}\label{transience}
    The Markov Chain is \textit{transient} if it is uniformly bounded, $\gamma_i<\infty$ for all $i\geq 0$, and $\liminf_{i \rightarrow \infty} \gamma_i >0$.
\end{lemma}

\begin{proof}
   The last condition implies that there exists $\epsilon>0$ and $N \geq 1$ such that $\gamma_i \geq \epsilon$ for~all $i \geq N$. The result then follows Theorem~\ref{fayolle-transience} taken with $f(i)=i$ for all $i$.
\end{proof}

\section{Behaviour around 0 and a condition for transience}

In the previous section, we have established mean drift conditions for the chain, in the cases when $\gamma_i$ is bounded away from 0 everywhere except for finitely many states. We have also seen that when the mean drift is \textit{exactly} zero for all but finitely many states, then the chain is null recurrent. 

The behaviour of the chain is unpredictable when the mean drifts are eventually non-negative and they converge to $0$, as can be seen from the examples below.

\begin{example}
    Let the chain be governed by the transition probabilities $p_{0i}=\frac{1}{2^i}$ for all $i \geq 1$, $p_{10}=p_{12}=\frac{1}{2}$ and $p_{i0}=\frac{1}{2}, p_{i,3i}=\frac{1}{4}, p_{i, i+1}= \frac{1}{8}+\frac{1}{i}, p_{i, i-1}=\frac{1}{8}-\frac{1}{i}$ for $i \geq 2$. Once again, observe that $\inf_{i \in \N} p_{i0}=1/2>0$, therefore, by Corollary~\ref{awayfromzero}, the chain is ergodic. However,  for $i \geq 2$,
    \begin{align*}
        \gamma_i&= 2i \cdot \frac{1}{4} +\frac{1}{8}+ \frac{1}{i} - \left(\frac{1}{8} -\frac{1}{i} \right) - i \cdot \frac{1}{2} = \frac{2}{i} \rightarrow 0  \quad \textbf{as } i \rightarrow \infty.
    \end{align*}
\end{example}

\begin{example}
    Consider a birth-death discrete process with transition probabilities $p_{i,i+1}=p_i$ and $p_{i,i-1}=q_i$. The condition \cite{karlin} for recurrence is then $\sum_{i=1}^\infty \prod_{j=1}^i \frac{q_j}{p_j} = \infty$ and $\sum_{i=1}^\infty \prod_{j=1}^i \frac{q_j}{p_j} < \infty$ for transience.

    \par Take
$p_i = \frac{1}{2} + \frac{1}{2(i+1)}, q_i = \frac{1}{2} - \frac{1}{2(i+1)}$. The mean drift at state $i$ is $\frac{1}{i+1} \to 0 \text{ as } n \to \infty$. Note that
\[
\rho_i := \prod_{j=1}^{i} \frac{q_j}{p_j} = \prod_{j=1}^{i} \frac{j}{j+2} = \frac{2}{(j+1)(j+2)}.
\]
Then $\sum_{i=0}^\infty \rho_i < \infty$, as each term is of order $i^{-2}$, 
so the chain is transient.
\par Now take $
p_i = \frac{1}{2} + \frac{1}{4(i+1)}, q_i = \frac{1}{2} - \frac{1}{4(i+1)}$. The mean drift is
$\gamma_i = p_i - q_i = \frac{1}{2(i+1)}  \to 0$. However, 
$\rho_i := \prod_{j=1}^{i} \frac{q_j}{p_j} = \prod_{j=1}^{n} \frac{2j+1}{2j+3} = \frac{3}{2j+3},$
so that
$
\sum_{i=0}^\infty \rho_i  = \infty$ and the chain is recurrent.
\end{example}

The above examples show that it is not so straightforward anymore to deduce the behaviour of the chain when the mean drifts are asymptotically zero. Lamperti \cite{lamperti} has shown some conditions for recurrence and transience of the chain around zero in terms of the \textit{second-moment drifts} $\sigma_i:=\expect((X_{n+1}-X_n)^2 \,|\, X_n=i )$. Esentially, Lamperti's criterion says that the chain is recurrent if $\gamma_i \leq \theta \frac{\sigma_i}{2i}$ for all large $i$ and $\theta<1$, and it is transient if $\gamma_i \geq \theta \frac{\sigma_i}{2i}$ for all large $i$ and $\theta>1$, provided that all $\sigma_i$ are bounded away from zero. \par Subsequent work on chains with asymptotically zero drift involve some sort of regularity assumptions on $\sigma_i$ \cite{denisov}. In this final section, we shall present a new transience criterion for uniformly bounded chains with asymptotically zero drifts, which makes no assumption on second-moment drifts.

\bigskip

\begin{theorem}[A transience criterion for locally regular monotone drifts]\label{thm:transience}
Let the chain be uniformly bounded with bound $d$. Suppose there exists $N \geq 1$ such that for all $i\geq N$:
\begin{enumerate}
  \item[(i)] $\gamma_i>0$;
  \item[(ii)] the sequence $(\gamma_i)_{i\ge N}$ is non-increasing;
  \item[(iii)] $\displaystyle\sum_{k=0}^\infty \gamma_k^2<\infty$
  \item[(iv)] the local oscillation $\displaystyle \delta_i:=\max_{|k-i|\le d}|\gamma_k^2-\gamma_i^2|$ satisfies
$ d\,\delta_i \le \tfrac{1}{2}\,\gamma_i^3.$
 
\end{enumerate}
The chain is then transient.
\end{theorem}

\begin{proof}
Define the Lyapunov function
\[
f(i):=\sum_{k=i}^\infty \gamma_k^2, \qquad i\ge0.
\]
By the third assumption, $f$ is bounded and strictly decreasing, hence
$f(d+1) < f(d) = \inf_{0\le j\le d} f(j)$. If we manage to prove that $\Delta f(i):=\sum_{j \in \N_0} p_{ij}(f(j)-f(i)) \leq 0$, then we can apply Theorem~\ref{transient-mertens} to conclude that the chain is transient. We have
\[
\Delta f(i)=\sum_{j<i} p_{ij}\sum_{k=j}^{i-1}\gamma_k^2 \;-\; \sum_{j>i} p_{ij}\sum_{k=i}^{j-1}\gamma_k^2.
\]
Because $(\gamma_k^2)$ is non-increasing, we have for $j<i$,
$\sum_{k=j}^{i-1}\gamma_k^2 \le (i-j)\gamma_j^2$,
and for $j>i$,
$\sum_{k=i}^{j-1}\gamma_k^2 \ge (j-i)\gamma_{j-1}^2.$ Then
\begin{equation*}
\label{incrementineq}
    \begin{split}
\Delta f(i)&\leq \sum_{j<i} p_{ij}(i-j)\gamma_j^2 \;-\; \sum_{j>i} p_{ij}(j-i)\gamma_{\,j-1}^2\\
&= \gamma_i^2\Big(\sum_{j<i}(i-j)p_{ij} - \sum_{j>i}(j-i)p_{ij}\Big)+ E_i\\
&=-\gamma_i^3+E_i,
\end{split}
\end{equation*}
where 
\[
E_i \;=\; \sum_{j<i} p_{ij}(i-j)\big(\gamma_j^2-\gamma_i^2\big)
\;-\; \sum_{j>i} p_{ij}(j-i)\big(\gamma_{j-1}^2-\gamma_i^2\big).
\]
Since $|i-j|\leq d$ when the probabilities are non-zero and the probabilities sum to at most $1$, we obtain the bound
$|E_i|\le d\cdot \max_{|k-i|\le d}|\gamma_k^2-\gamma_i^2| = d\,\delta_i$. By condition~$(iv)$, we have $d\,\delta_i\le \tfrac12\gamma_i^3$, hence for all $i\geq N$,
\[
\Delta f(i) \le -\gamma_i^3 + d\,\delta_i \le -\tfrac12 \gamma_i^3 \leq  0.
\]
Therefore, the conditions of Theorem~\ref{transient-mertens} are satisfied and the chain is transient.
\end{proof}

We now show in the following corollary how to construct a family of drifts satisfying the four conditions of the theorem above.

\begin{corollary}
\label{driftfamily}
    Suppose there exist constants $C>0$ and $\alpha\in(\tfrac12,1)$ and a sequence $\varepsilon_i\to0$ such that
\begin{equation}\label{eq:asymp}
\gamma_i = C\, i^{-\alpha}\big(1+\varepsilon_i\big),\qquad i\ge 1,
\end{equation}
and $ \max_{|h|\le d} |\varepsilon_{i+h}|=O(i^{-1})$. Assume moreover that $(\gamma_i)$ is eventually positive and nonincreasing. Then the conditions of the above theorem are satisfied, and, consequently, a chain with mean drifts $(\gamma_i)_{i \geq 0}$ is transient.
\end{corollary}

\begin{proof}
Since, by assumption, conditions~$(i)$~and $(ii)$ are satisfied, we only need to check the last two conditions. 
\par The square summability follows because there exists $i_0 \geq 1$ such that $|\varepsilon_i|<1$, and consequently 
$\gamma_i^2 = C^2 i^{-2\alpha} (1+\varepsilon_i)^2 \leq 4C^2i^{-2\alpha}$, for all $i \geq i_0$.
Since \(2\alpha>1\), the series \(\sum_i i^{-2\alpha}\) converges, and hence \(\sum_i \gamma_i^2<\infty\).
\par To check condition~$(iv)$, fix $i$ and $h$ with \(|h|\leq d\). Using \eqref{eq:asymp} we write
\begin{equation}
\label{gammabound}
    \begin{split}
    \gamma_{i+h}^2 - \gamma_i^2&= C^2\Big[(i+h)^{-2\alpha}(1+\varepsilon_{i+h})^2 - i^{-2\alpha}(1+\varepsilon_i)^2\Big]\\
    &= C^2\big((i+h)^{-2\alpha} - i^{-2\alpha}\big) + R_{i,h},
    \end{split}
\end{equation}
where 
\[
R_{i,h} = C^2\Big[(i+h)^{-2\alpha}\big((1+\varepsilon_{i+h})^2-1\big) - i^{-2\alpha}\big((1+\varepsilon_i)^2-1\big)\Big].
\]
For  $i\geq i_0+d$, we have that $|\varepsilon_j|\le 1$ for all $j\in \{i-d,\dots,i+d\}$. Then
$|(1+\varepsilon_j)^2-1| \le 2|\varepsilon_j| + \varepsilon_j^2 \le 3|\varepsilon_j|.$
Hence, for all $|h|\le d$,
\[
|R_{i,h}| \le 3 C^2 \Big[ (i+h)^{-2\alpha} |\varepsilon_{i+h}| + i^{-2\alpha} |\varepsilon_i| \Big] \le 3 C^2 \big((i+h)^{-2\alpha} + i^{-2\alpha}\big) \max_{|i-j|\leq d} |\varepsilon_j|,
\]
Since $|h|\le d$, we have, for $i \geq 2d$, $(i+h)^{-2\alpha} \le (i-d)^{-2\alpha} \le (i/2)^{-2\alpha} = 2^{2\alpha} i^{-2\alpha}$. Therefore
\begin{equation}
\label{rih-bound}
    |R_{i,h}| \le 3 C^2 (2^{2\alpha}+1) i^{-2\alpha} \max_{|i-j|\leq d} |\varepsilon_j|.
\end{equation}

\par Now, using the Mean Value Theorem, we have
\[
\big|(i+h)^{-2\alpha} - i^{-2\alpha}\big| = 2\alpha |h|\, \xi^{-(2\alpha+1)},
\]
where $\xi \in [i, i+h]$ with $|h| \leq d$. Therefore, for $i \geq 2d$, we have

\begin{equation}
\label{ibound}
\begin{split}
    \big|(i+h)^{-2\alpha} - i^{-2\alpha}\big|& \leq 2\alpha |h|\, (i-d)^{-(2\alpha+1)}\\
    &\le 2\alpha |h|\,(i/2)^{-(2\alpha+1)}
   = \bigl(2\alpha \cdot 2^{\,2\alpha+1}\bigr)\,{|h|}{i^{-(2\alpha+1)}}.
\end{split}
\end{equation}

Therefore, for $i \geq \max\{2d, i_0+d\}$, we obtain by plugging~\eqref{rih-bound} and~\eqref{ibound} into~\eqref{gammabound},
\[
\big|\gamma_{i+h}^2 - \gamma_i^2\big| \le C^2(2\alpha \cdot 2^{\,2\alpha+1}\bigr)  {|h|}{i^{-(2\alpha+1)}} + 3 C^2 (2^{2\alpha}+1) i^{-2\alpha} \max_{|i-j|\leq d} |\varepsilon_j|.
\]
Taking the maximum over \(|h|\le d\) yields
\[
\delta_i \le C^2(2\alpha \cdot 2^{\,2\alpha+1}\bigr)  {d}\,{i^{-(2\alpha+1)}} + 3 C^2 (2^{2\alpha}+1) i^{-2\alpha} \max_{|i-j|\leq d} |\varepsilon_j|.
\]
Using that \(\max_{|u|\le d}|\varepsilon_{i+u}|=O(i^{-1})\), we have that there exists $i_1$ such that for all $i \geq i_1$, $\max_{|u|\le d}|\varepsilon_{i+u}| \leq D \,i^{-1}$ for some $D>0$ . Therefore, for $i \geq \max \{i_0+d, 2d, i_1\}$, we have 
\[
\delta_i \le A\, i^{-(2\alpha+1)}
\]
where $A:=C^2(2\alpha \cdot 2^{\,2\alpha+1}\bigr)d+3 C^2 (2^{2\alpha}+1)D$.

\par We now compare $\delta_i$ with $\gamma_i^3$. Observe that there exists $i_2$ such that $\gamma_i^3 \geq \frac{1}{2}C^3 i^{-3\alpha}$ for~all~$i \geq i_2$. Therefore, for $i \geq \max\{i_0+d, 2d, i_1, i_2\}$, we have 

\[
\frac{d\,\delta_i}{\gamma_i^3}
\leq \frac{2d A\, i^{-(2\alpha+1)}}{C^3 i^{-3\alpha}}
= \frac2{d A}{C^3}\, i^{\alpha-1} \rightarrow 0 \quad \text{as } i \rightarrow \infty.
\]
Thus,  there exists \(N\) such that for all \(i\geq N\),
\[
d\,\delta_i \le \tfrac12\,\gamma_i^3.
\]
This is precisely condition~$(iv)$. We conclude that the mean drifts satisfy all conditions of Theorem~\ref{thm:transience} and, consequently, a chain with such mean drifts is transient.

\end{proof}

We now conclude the present paper with showing a concrete example in which the above criterion can be applied, but Lamperti-type criteria fail.

\begin{example}
    Let $\frac{1}{2}<\alpha<1$ and consider the chain governed by the transition probabilities $p_{i,i-2}=\frac{1}{9}i^{-\alpha}$, $p_{i,i-1}=\frac{1}{4}i^{-\alpha}$, $p_{i,i+1}=\frac{1}{2}i^{-\alpha}$, and $p_{i,i}=1-\frac{31}{36}i^{-\alpha}$. The drift at state $i$ is then given by 
    \[
    \gamma_i=\left(\frac{1}{2}-\frac{1}{4}-\frac{2}{9}\right)i^{-\alpha}=\frac{1}{36}i^{-\alpha},
    \]
    which, using Corollary~\ref{driftfamily}, gives that the chain is transient. 
    \par However, computing $\sigma_i$ gives
    \[
    \sigma_i=\left(\frac{4}{9}+\frac{1}{4}+\frac{1}{2}\right)i^{-\alpha} = \frac{43}{36}i^{-\alpha} \, \rightarrow \, 0 \quad \text{as } i \rightarrow \infty. 
    \]
    Therefore, the second-moment drifts are not bounded away from $0$ and Lamperti's criterion can not be applied.
\end{example}

\end{document}